\documentclass[12pt]{amsart}

\usepackage[all]{xy}
\usepackage{amssymb}
\usepackage{amsmath}
\usepackage{stmaryrd}
\usepackage{bbm}
\usepackage{txfonts}
\usepackage{tikz-cd}
\usepackage{url}
\usepackage{etoolbox}

\makeatletter
\patchcmd{\@addmarginpar}{\ifodd\c@page}{\ifodd\c@page\@tempcnta\m@ne}{}{}
\makeatother
\reversemarginpar

\usepackage[OT2,T1]{fontenc}
\DeclareSymbolFont{cyrletters}{OT2}{wncyr}{m}{n}
\DeclareMathSymbol{\Sha}{\mathalpha}{cyrletters}{"58}

\theoremstyle{definition}
\newtheorem*{Definition*}{Definition}

\theoremstyle{definition}
\newtheorem*{Proposition*}{Proposition}

\theoremstyle{definition}
\newtheorem*{Theorem*}{Theorem}

\theoremstyle{definition}
\newtheorem*{Corollary*}{Corollary}

\theoremstyle{definition}
\newtheorem*{Conjecture*}{Conjecture}

\theoremstyle{definition}

\theoremstyle{definition} 
\newtheorem{sProposition}[subsection]{Proposition}

\theoremstyle{definition} 
\newtheorem{sDefinition}[subsection]{Definition}

\theoremstyle{definition}

\theoremstyle{definition}

\theoremstyle{definition}

\theoremstyle{definition}

\theoremstyle{definition} 
\newtheorem{sTheorem}[subsection]{Theorem}

\theoremstyle{definition}

\theoremstyle{definition}

\theoremstyle{definition}

\theoremstyle{definition}

\theoremstyle{definition} 
\newtheorem{sRemark}[subsection]{Remark}

\theoremstyle{definition}

\theoremstyle{definition}

\theoremstyle{definition}

\theoremstyle{definition}

\let\oldmarginpar\marginpar
\renewcommand\marginpar[1]{\-\oldmarginpar[\raggedleft\footnotesize #1]%
{\raggedright\footnotesize #1}}

\newcommand{\Rep}{\operatorname{\bf{Rep}}}
\newcommand{\Vect}{\operatorname{\bf{Vect}}}

\newcommand{\xto}{\xrightarrow}

\newcommand{\gr}{\operatorname{gr}}

\newcommand{\Spec}{\operatorname{Spec}}

\newcommand{\Hom}{\operatorname{Hom}}

\newcommand{\Isom}{\operatorname{Isom}}

\newcommand{\Aut}{\mathrm{Aut}\,}

\newcommand{\Mod}{\operatorname{Mod}}

\newcommand{\m}[1]{\mathrm{#1}}
\newcommand{\fk}[1]{\mathfrak{#1}}

\newcommand{\bb}[1]{\mathbb{#1}}

\newcommand{\si}{\sigma}

\newcommand{\ze}{\zeta}
\newcommand{\ga}{\gamma}

\newcommand{\om}{\omega}

\newcommand{\ep}{\epsilon}

\newcommand{\Gm}{{\mathbb{G}_m}}

\newcommand{\CC}{\bb{C}}
\newcommand{\GG}{\mathbb{G}}

\newcommand{\QQ}{\bb{Q}}

\newcommand{\PP}{\bb{P}}
\newcommand{\pP}{\fk{p}}

\newcommand{\Uu}{\mathcal{U}}
\newcommand{\uU}{\mathfrak{u}}
\newcommand{\cC}{\mathfrak{c}}

\renewcommand{\AA}{\bb{A}}

\newcommand{\Oo}{\mathcal{O}}

\newcommand{\Aa}{\mathcal{A}}

\newcommand{\inv}{^{-1}}

\newcommand{\areq}{\ar@{=}}
\newcommand{\suphook}{\ar@{^(->}}
\newcommand{\subhook}{\ar@{_(->}}

\newcommand{\smses}[6]
{
\[
\xymatrix{
1 \ar[r] &
#1 \ar[r]_-{#2} &
#3 \ar[r]_-{#4} &
#5 \ar[r] \ar@/_1.5pc/[l]_-{#6} &
1
}
\]
}

\newcommand{\thrpl}{\PP^1 \setminus \{0,1,\infty\}}

\newcommand{\dR}{{\rm {dR}}}

\newcommand{\per}{\operatorname{per}}

\newcommand{\un}{\m{un}}

\newcommand{\one}{\mathbbm{1}}

\newcommand{\can}{\m{can}}

\newcommand{\opnm}{\operatorname}

\newcommand{\new}{\newcommand}

\new{\parnonumber}{\medskip\noindent }

\new{\spar}
[1]{\medskip\refstepcounter{subsection}\noindent\textbf{\arabic{section}.\arabic{subsection}.}\label{#1}}

\new{\sspar}
[1]{\medskip\refstepcounter{subsection}\noindent\arabic{section}.\arabic{subsection}.\arabic{subsubsection}.\label{#1}}

\new{\oo}{$\infty$}

\new{\Aaug}{\opnm{\Aa\m{ug}}}

\new{\MYM}{\opnm{M_YM}}

\new{\MTM}{\opnm{MTM}}

\new{\Emb}{\opnm{Emb}}

\new{\cber}{\fk{c}_\m{Ber}}

\new{\Cr}{\opnm{Cr}}

\new{\Inv}{\opnm{Inv}}

\new{\pber}{\pP_\m{Ber}}

\new{\Fratila}{Fr\u{a}\c{t}il\u{a} }

\new{\TM}{\opnm{TM}}

\new{\comment}[1]{}

\title[On Andr\'e periods]{On Andr\'e periods of mixed Tate motives}

\author{Ishai Dan-Cohen}
\thanks{\textbf{Grant Acknowledgement:} This work was supported by ISF grant number 621/21. \textbf{Political Disclaimer:} The author asks not to be considered responsible for the actions of any government that does not fully embrace the principles of democracy.}

\date{\today}

\begin{document}

\maketitle


\raggedbottom
\SelectTips{cm}{11}

\begin{abstract}
In this note, we show that the $p$-adic periods of motives introduced recently by Ancona and \Fratila (``Andr\'e periods'') reduce to the classically studied notion in the case of Mixed Tate motives in connection with $p$-adic multiple zeta values and $p$-adic multiple polylogarithms. We also connect Andr\'e periods with Coleman integration by observing that the Frobenius-fixed de Rham paths of Besser and Vologodsky come from motivic paths in characteristic $p$ (unconditionally in the mixed Tate setting, conditionally in general). We use this to realize special values of $p$-adic multiple polylogarithms as Andr\'e periods in a concrete way.  

\medskip


\medskip


\end{abstract}

\section{Introduction}

The category of mixed Tate motives over an open integer ring\footnote{By this we mean a finite localization of the ring of integers in a number field.}
or a number field possesses a notion of \emph{$p$-adic period} which differs from the $p$-adic periods of $p$-adic Hodge theory. At first glance, this notion may appear to be a peculiarity of the mixed Tate setting.
Yet it plays a key role in the inner mechanisms of mixed Tate Chabauty-Kim theory. It also connects the study of $p$-adic iterated integrals with Goncharov's theory of \textit{motivic iterated integrals}, and allows one to investigate Goncharov's conjectures from a $p$-adic point of view. Lastly, it forms the basis for the so-called $p$-adic period conjecture.\footnote{For mixed Tate motives, see, for instance \cite{DelGon, BurgosFresan, dupont2024introduction}. For $p$-adic periods of mixed Tate motives, see, for instance \cite{YamashitaBounds, chatzistamatiou2011p, FurushoPMZVII}. For the applications discussed above, see, for instance, \cite{CKtwo, mtmue, mtmueii, PolGonI, PolGonII, DCJar, JarossaySaettoneZehavi, ludtke2025refinedZ16, BettsLudtkeRefined, BettsLudtkeSection}.}

In this theory (and in stark contrast to the classical theory of complex periods, as well as $p$-adic Hodge theory), the period of the Tate motive vanishes.\footnote{See Remark \ref{RM55}.}
In recent work \cite{DCCorwinPer}, D. Corwin and I develop a gadget that extends the $p$-adic period map beyond the mixed Tate setting. As in the mixed Tate case, our gadget assigns naught to pure motives (which must now be replaced by certain systems of realizations).\footnote{See Definition 5.6 of \cite{DCCorwinPer}; Remark \ref{RM55} below applies with little change.}
As such, it interacts usefully with cocycles for unipotent fundamental groups, and helps to begin the process of extending the mixed Tate Chabauty-Kim theory of \cite{mtmue, mtmueii, brown2017integral, PolGonI, PolGonII} (and others) beyond the mixed Tate setting. 

However, the literature has long since contained various notions of ``$p$-adic periods'' of a similar sort, and in this setting, it is not considered to be the case that $p$-adic periods of all pure motives are trivial.
These notions have finally been put on solid ground by Ancona and \Fratila  \cite{AnconaFratila}, who drew inspiration especially from work of Y. Andr\'e. 

Ancona and \Fratila work primarily in the setting of pure motives, where various concrete candidates for the correct Tannakian category exist. Working with \textit{mixed} motives means adding a conjectural layer, which they do only as an afterthought. However, \textit{Mixed Tate motives} provide a microcosm, essentially orthogonal to the setting of pure motives, in which the correct Tannakian category \textit{is} available, at least over number fields and finite fields. Our purpose in this note is to show that their notion too extends the classical notion of $p$-adic periods of mixed Tate motives.\footnote{See Section \ref{SecRev} for a review of the latter.} 

Recall, for instance, that if $\MTM(Z)$ is the category of mixed Tate motives over an open integer scheme $Z \subset \Spec \Oo_K$ with fraction field $K$,\footnote{By this we mean an open subscheme of $\Spec \Oo_K$ for $K$ a number field.} then the complex period map
\[
\per^\CC: \Oo P(Z) \to \CC
\] 
is defined on the coordinate ring of a certain principal homogeneous space $P(Z)$. In this setting, elements of the coordinate ring $\Oo P(Z)$ are sometimes referred to as ``formal periods'', and so the period map assigns to a formal period its \textit{actual} period. 

Let 
\[
B^*: \MTM(Z) \to \Vect \QQ
\]
be the Betti realization functor associated to a choice of $\CC$-valued point of $Z$, and let
\[
\dR^*: \MTM(Z) \to \Vect K
\]
be the de Rham realization functor. Then (in one formulation)
\[
P(Z):= \Isom^\otimes(B^*_K, \dR^*) 
\]
is the Tannakian path torsor, integration gives rise to a complex point
\[
\Spec \CC \to P(Z),
\]
and $\per^\CC$ is the induced map of coordinate rings. 

Turning to a local field $K_\pP$ of $K$ at a prime $\pP \in Z$, the classical theory of $p$-adic periods of mixed Tate motives is based on a map of rings
\[
\per^\m{Cl}: \Oo U_\dR(Z) \to K_\pP,
\]
while the theory of Ancona-\Fratila (translated to our setting) is based on a map 
\[
\per^\m{AF}: \Oo H(Z) \to K_\pP.
\]
Both involve a principal homogeneous space ($U_\dR(Z)$, $H(Z)$) which stands in imperfect analogy to the earlier $P(Z)$. Moreover, the respective period maps involve apparently different ingredients. Roughly speaking, $\per^\m{Cl}$ compares the Hodge filtration and the Frobenius automorphism on the $p$-adic de Rham realization, whereas $\per^\m{AF}$ compares the rational de Rham realization with the crystalline realization, with emphasis on the factorization of the latter through the category of mixed Tate motives over the residue field at $\pP$. The main observation of this paper is that the two are nevertheless equivalent:

\begin{sTheorem}
\label{A7}
There exists an isomorphism of $K$-algebras of ``formal periods of mixed Tate motives'' commuting with the two period maps, as in the following diagram:
\[
\xymatrix
@ R = 2pt
{
\Oo H(Z) \ar[dd]_\sim \ar[dr]^-{\per^\m{AF}}
\\
& K_\pP \;.
\\
\Oo U_\dR(Z) \ar[ur]_-{\per^\m{Cl}}
}
\]
\end{sTheorem}

A priori, the $K$-algebra $\Oo H(Z)$ of ``formal Ancona-\Fratila periods of mixed Tate motives'' depends on the prime $\pP$. However, the prounipotent group $U_\dR(Z)$, which is just the prounipotent radical of the Tannakian Galois group
\[
G_\dR(Z) = \Aut^\otimes(\dR^*)
\]
does \textit{not} depend on $\pP$. Thus, one noteworthy consequence of theorem \ref{A7} is that $H(Z)$ is independent of $\pP$ up to canonical isomorphism.

Interesting examples of mixed Tate motives arise especially from the \textit{unipotent fundamental group} of the projective line with punctures
\[
\PP^1 \setminus \{a_1, \dots, a_n\}
\]
and its path torsors, and interesting examples of $p$-adic periods of mixed Tate motives come from those. These take the form of $p$-adic iterated integrals (known variously also as \textit{Coleman} or \textit{Coleman-Besser functions}) and include special values of \textit{$p$-adic multiple polylogarithms}, among them the highly studied \textit{$p$-adic multiple zeta-values}. We show that all such $p$-adic numbers arise as \textit{Andr\'e $p$-adic periods} (as Ancona and \Fratila call them) in a natural and concrete way. In particular, we observe that the Frobenius-fixed de Rham paths of Besser \cite{Besser} and Vologodsky \cite{Vologodsky} are motivic after all (in contradiction of an oft-repeated adage), as long as one replaces motives over an arithmetic base by motives over one of its residue fields. 

\subsection{Notations and conventions}
\label{Not}
We work throughout with a number field $K$ with ring of integers $\Oo_K$, and an open integer scheme
\[
Z \subset \Spec \Oo_K.
\]
We also fix a closed point $\pP \in Z$, we denote the local field at $\pP$ by $K_\pP$, the ring of integers of $K_\pP$ by $\hat \Oo_\pP$, and the residue field at $\pP$ by $k$. Finally, we let $K_\pP^0$ denote the fraction field of the ring of Witt vectors of $k$.

We use an upper sharp ($\sharp$) to denote pullback of functions along a morphism of schemes or along a morphism of presheaves of sets, to be proved representable by schemes later.

\subsection*{Acknowledgements}
I wish to thank Giuseppe Ancona and Drago\c{s} \Fratila for helpful conversations and email exchanges, and Martin L\"udtke for helpful comments on an early draft of this manuscript. I wish to thank the referees for their careful reading and helpful comments.

\section{Review of Andr\'e periods}

We begin by transporting the construction of Ancona-\Fratila to the setting of mixed Tate motives over a number field. 

\spar{A1}
We put ourselves in the situation of paragraph \ref{Not}. We then have the categories of mixed Tate motives $\MTM(Z)$, $\MTM(k)$ available, as well as a pullback functor (``specialization'')
\[
\tag{*}
s_\pP^* : \MTM(Z) \to \MTM(k).
\]
(See, for instance, Deligne-Goncharov \cite{DelGon} for a summary of the theory, which goes back to work of Levine and Voevodsky). Let $\dR^*$ denote the de Rham realization functor
\[
\dR^*: \MTM(Z) \to \Vect K.
\]
Let $\Cr^*$ denote the crystalline realization functor 
\[
\Cr^*:\MTM(k) \to \Vect K_\pP^0. 
\]
Let $\ze$ denote the lax-monoidal functor $\MTM(k) \to \Vect \QQ$
\[
\tag{**}
\ze(M) = \Hom\big(\QQ(0), M \big).
\]

\spar{extscal}
If $R'$ is an $R$-algebra, we let $e_{R'/R}$ denote \textit{extension of scalars} 
\[
\tag{*}
\Mod(R) \to \Mod(R')
\]
\[
V \mapsto R' \otimes_R V.
\]
Let $H(Z): \opnm{Alg}(K) \to \mathbf{Set}$ be the functor 
\[
H(Z)(R) = \Hom^\otimes \big( e_{R/\QQ} \circ \ze \circ s_\pP^*, e_{R/K} \circ \dR^*),
\]
In other words, $H(Z)$ is the functor of $\otimes$-compatible natural transformations, which we'll allow ourselves to write
\[
\tag{**}
H(Z) = \Hom^\otimes (\ze \circ s_\pP^*, \dR^*)
\]
without adding an underline. We view $H(Z)$ as a presheaf on the category of affine $K$-schemes. \footnote{As we'll see in section \ref{SecComp} (and as proved in a slightly different setting by Ancona-\Fratila), $H(Z)$ is a trivial torsor under a prounipotent group over $K$, hence in particular representable by a $K$-scheme.} Following Ancona-\Fratila \cite{AnconaFratila}, we define the \emph{ring of $p$-adic formal Andr\'e periods of mixed Tate motives over $Z$} to be the $K$-algebra
\[
\Oo H(Z) = \Hom \big(H(Z), \AA^1_K \big).
\]

\begin{sDefinition}
\label{B2}
Recall that $K_\pP$ denotes the local field of $K$ at $\pP$, a finite, purely ramified, extension of the absolutely unramified $p$-adic field $K_\pP^0$. In terms of the specialization functor $s_\pP^*$ (\ref{A1}*), the de Rham realization functor $\dR^*$, the functor $\ze$ (\ref{A1}**), and the extension of scalars functors $e_{?/??}$ (\ref{extscal}*), we define a $\otimes$-embedding
\[
\cber: e_{K_\pP/\QQ} \circ \ze \circ s_\pP^* 
\to 
e_{K_\pP/K} \circ \dR^*
\]
(hence a $K_\pP$-point of the functor $H(Z)$ (\ref{extscal}**)) as follows. Fixing $M \in \MTM(Z)$, we are to define a map of $K_\pP$-\textit{vector spaces}
\[
\cber(M):
K_\pP \otimes_\QQ \Hom \big( \QQ(0), s_\pP^* M \big)
\to
K_\pP \otimes_K \dR^*M.
\]
For this, fix a morphism $v: \QQ(0) \to s_\pP^* M$ in $\MTM(k)$. Then $\Cr^* v$ is a map of $K_\pP^0$-\textit{vector spaces}
\[
\tag{*}
\Cr^* v: K_\pP^0 \to \Cr^* s_\pP^* M.
\]
The Berthelot comparison isomorphism \cite[Theorem V.2.3.2]{berthelot2006cohomologie} gives rise to an isomorphism
\[
\tag{**}
\pP^\m{Ber}_M: K_\pP \otimes_{K_\pP^0} \Cr^* s_\pP^* M \xto{\sim} 
K_\pP \otimes_K \dR^*M
\]
(see Theorem 9.1, Notation 9.1 and Remark 9.3 of \cite{AnconaFratila} and the reference there). We define
\[
\cber(M)(v) := \big( \pP^\m{Ber}_M \circ \Cr^* v \big)(1).
\]
Viewed as a $K_\pP$-point of $H(Z)$, we refer to $\cber$ as the \emph{Berthelot point}.

Recall that we use an upper sharp to denote pullback of functions (\ref{Not}).
In terms of the $K_\pP$-point $\cber\in H(Z)(K_\pP)$, we define the \emph{Ancona-\Fratila period map} by
\[
\tag{$\dagger$}
\per^\m{AF} :=
\cber^\sharp : \Oo H(Z) \to K_\pP.
\]
We define the ring of \emph{Andr\'e periods of mixed Tate motives over $Z$} to be the image of $\per^\m{AF}$. 
\end{sDefinition}

\medskip
This is the definition given by Ancona-\Fratila \cite{AnconaFratila}, transported to our slightly different setting. We have not made any meaningful use of the very special (and simple) structure of the category of mixed Tate motives (it's a ``mixed Tate category''). As we'll see below, we can in fact construct $\cber$ without reference to $\MTM(k)$. 

\spar{A2}
\textbf{Variant (work over $\overline \Oo_\pP$).}
Alongside the above, we may also consider the following setup. We let $\overline \Oo$ be the ring of algebraic integers, $\pP$ a prime of $\overline \Oo$, $\overline \Oo_{\pP}$ the local ring at $\pP$, $k=\overline \Oo/\pP$ the residue field at $\pP$. If
\[
\overline Z = \Spec \overline \Oo_\pP,
\]
we may again consider the associated categories of mixed Tate motives, as well as the pullback map
\[
s_\pP^* : \MTM\big(\overline Z\big) \to \MTM(k).
\]
In this case, these are merely conjectural. Assuming expected properties, analogous constructions provide analogous objects ($\dR^*$, $\Cr^*$, $\ze$, $H\big(\overline Z\big)$, $\Oo H\big(\overline Z\big)$, $\cber$, $\per^\m{AF}$). 

\spar{A4}
\textbf{Variant (Restrict to principal subcategories).}
In another variation, we may fix a mixed Tate motive $M \in \MTM(Z)$ over $Z$, replace $\MTM(Z)$ by the Tannakian subcategory
\[
\langle M \rangle \subset \MTM(Z)
\]
generated by $M$, and replace $\MTM(k)$ by the Tannakian subcategory
\[
\langle M_p \rangle \subset \MTM(k)
\]
generated by
\[
M_\pP := s_\pP^* M.
\]
Analogous constructions provide analogous objects ($\dR^*|_{\langle M \rangle}$, $\Cr^*|_{\langle M \rangle}$, $\ze|_{\langle M \rangle}$, $H(\langle M \rangle)$, $\Oo H(\langle M \rangle)$, $\cber$, $\per^\m{AF}$); we speak of \emph{Andr\'e periods of $\langle M \rangle$}.\footnote{Ancona-\Fratila also discuss a potentially smaller ring of \textit{Andr\'e periods of $M$} \cite[Definition 5.5]{AnconaFratila}.}

\spar{A5}
\textbf{Variant (Principal subcategories over $\overline \Oo_\pP$).}
We combine variants \ref{A2} and \ref{A4}. Section 9 of Ancona-\Fratila \cite{AnconaFratila} takes place in this setting.

\section{Classical $p$-adic periods of mixed Tate motives}
\label{SecRev}

We recast the classical construction in the language of joint work with D. Corwin \cite{DCCorwinPer}. In the present setting, this is merely a terminological variation on the construction of \cite[\S4]{mtmue} which is better suited to generalization to bigger categories of motives. 

\spar{A3}
In the presence of a Tannakian category $T$ and a fiber functor $\om$, we denote the associated Tannakian Galois group by $G_\om(T)$. Instead of writing $G_\om\big(\MTM(Z) \big)$ and $G_\om \big(\MTM(k) \big)$ we write $G_\om(Z)$ and $G_\om(k)$. The full subcategory $\TM(Z) \subset \MTM(Z)$ of semisimple objects induces a surjection of Tannakian fundamental groups, hence a short exact sequence of group schemes over $K$
\[
\tag{*}
1 \to U_\dR(Z) \to G_\dR(Z) \to \GG_{m, K} \to 1
\]
with $U_\dR(Z)$ prounipotent. 

Categories of mixed Tate motives have canonical $\QQ$-rational fiber functors $\om_\m{can}$ given (in each setting) by
\[
\om_\m{can}(M) = 
\bigoplus_i \Hom \big( \QQ(i), \gr^W_i M \big).
\]
The associated exact sequence $(*)_\m{can}$ of group schemes over $\QQ$ is canonically split. The base extension $(\om_\m{can})_K$ may undangerously be identified with the \emph{graded de Rham fiber functor} $\dR^* \circ \gr^W$.

\spar{A6}
Adding Hodge filtrations, the de Rham fiber functor factors through a symmetric monoidal functor
\[
\dR^*_\m{Fil}: \MTM(Z) \to \opnm{Mod}^\m{Fil} K
\]
to the category of filtered vector spaces. In the language of \cite{DCCorwinPer}, $\dR^*_\m{Fil}$ is a \emph{Hodge filtered fiber functor} on a \emph{weight filtered Tannakian category}. Consequently, there's a unique\footnote{The uniqueness here follows from the fact that in the present setting, in the notation of loc. cit. $F^0 U_\dR(Z) = \{1\}$.} 
Tannakian path
\[
\tag{*}
\pP^H: \dR^* \circ \gr^W \to \dR^*
\]
(the ``arithmetic Hodge path'') preserving Hodge filtrations and lifting the identity endomorphism of $ \dR^*_\m{Fil} \circ \gr^W$. In concrete terms, for each $M \in \MTM(Z)$,
\[
\pP^H_M: \gr^W M_\dR \to M_\dR
\]
is a splitting of the weight filtration which preserves Hodge filtrations. 

\spar{B3}
 Similarly, the crystalline fiber functor factors through a symmetric monoidal functor
\[
\Cr^*_\phi: \MTM(k) \to \opnm{Mod}^\phi K_\pP^0
\]
to the category of $\phi$-modules (vector spaces equipped with a Frobenius-linear automorphism --- in reference to the unique lifting of Frobenius to $K_\pP^0$). In the language of loc. cit., $\Cr^*_\phi$ is a \emph{Frobenius-equivariant fiber functor}. Consequently, there's a unique Tannakian path
\[
\tag{*}
\pP^\m{Cr}: \Cr^* \circ \gr^W \to \Cr^*
\]
(the ``arithmetic crystalline path'') preserving crystalline Frobenii and lifting the identity on $\Cr^* \circ \gr^W$. Transporting $\pP^\m{Cr}$ along the specialization functor $s^*_\pP$ and using the Berthelot comparison isomorphism, we obtain a Tannakian path
\[
\tag{**}
e_{K_\pP/K} \circ \dR^* \circ \gr^W 
\to 
e_{K_\pP/K} \circ \dR^*
\]
which we will denote (somewhat inaccurately) by $s^\pP_* \pP^\m{Cr}_{K_\pP}$. 
If $M \in \MTM(Z)$, then
\[
(s^\pP_*\pP^\m{Cr}_{K_\pP})_M: \gr^W M_\dR \to M_\dR
\]
is a splitting of the weight filtration which preserves Frobenii. 

\begin{sDefinition}
\label{B4}
In terms of the arithmetic Hodge path $\pP^H$ (\ref{A6}*), the arithmetic crystalline path $\pP^\m{Cr}$ (\ref{B3}), and the specialization functor $s_\pP^*$ (\ref{A1}*), we define the \emph{classical $p$-adic period loop}
\[
\uU: \Spec K_\pP \to U_\dR(Z) 
\]
by
\[
\uU =  s^\pP_* \pP^\m{Cr}_{K_\pP} \circ (\pP^H_{K_\pP})\inv,
\]
we define the \emph{classical $p$-adic period map} by 
\[
\tag{*}
\per^\m{Cl} := \uU^\sharp: \Oo U_\dR(Z) \to K_\pP,
\] 
and we define the \emph{ring of classical $p$-adic periods of mixed Tate motives over $Z$} to be the image of $\per^\m{Cl}$. 
\end{sDefinition}

This is just a reformulation of the notion considered by Yamashita \cite{YamashitaBounds}, Chatzistamatiou-\"Unver \cite{chatzistamatiou2011p}, and others.

\begin{sRemark}
\label{RM55}
If $E \in \MTM(Z)$ is a mixed Tate motive, $e \in E_\dR$ is a vector, and $\ep \in E_\dR^\lor$ is a linear functional, then the data $[E, e, \ep]$ (a ``Tannakian matrix entry'') gives rise to a function $G_\dR(Z) \to \AA^1_K$. Let $R$ be a $K$-algebra and consider an $R$-valued point $\ga \in G_\dR(Z)(R)$. Then $\ga_E$ is an invertible $R$-linear endomorphism of $R \otimes_K E_\dR$, and $[E,e,\ep](\ga)$ is defined to be the image of $1 \otimes e$ under the composition
\[
R \otimes_K E_\dR \xto{\ga_E} R \otimes_K E_\dR \xto{\ep} R.
\]
In the setting of mixed Tate motives, the \emph{$p$-adic periods of $E$} are, by definition, the values $[E, e, \ep](\uU) \in K_\pP$ for varying $e$ and $\ep$. In particular, if $E$ is semisimple, then $[E, e, \ep]$ vanishes on $U_\dR(Z)$, so $E$ has no nonzero periods.  
\end{sRemark}

\section{Comparison}
\label{SecComp}

We continue with our fixed number field $K$ and the open integer scheme $Z$ (\ref{Not}). As indicated in the introduction, in terms of the presheaf $H(Z)$ (\ref{extscal}**), the Ancona-\Fratila period map $\per^\m{AF}$ (\ref{B2}$\dagger$), the prounipotent group $U_\dR(Z)$ (\ref{A3}*), and the classical period map $\per^\m{Cl}$ (\ref{B4}*), our main observation is the following theorem.

\begin{Theorem*}[\ref{A7}]
There exists an isomorphism of $K$-algebras of ``formal periods'' commuting with the two period maps, as in the following diagram:
\[
\xymatrix
@ R = 2pt
{
\Oo H(Z) \ar[dd]_\sim \ar[dr]^-{\per^\m{AF}}
\\
& K_\pP \;.
\\
\Oo U_\dR(Z) \ar[ur]_-{\per^\m{Cl}}
}
\]
\end{Theorem*}

We will take a somewhat circuitous path to the proof of theorem \ref{A7} in order to have portions of the construction that generalize readily beyond the mixed Tate setting. In outline, we will show first that $H(Z)$ becomes a canonically trivial torsor under $U_\dR(Z)$ after base-change to $K_\pP$. It then follows that $H(Z)$ is a (necessarily trivial) torsor under $U_\dR(Z)$ already over $K$. Incidentally, this will show that $H(Z)$ is representable by an affine $K$-scheme.\footnote{Ancona-\Fratila show (in a more general setting) that for any $M \in \MTM(Z)$, $H(\langle M \rangle)$ is representable by a finite type affine $K$-scheme.} Finally, we'll construct a $K$-point of $H(Z)$ whose orbit map interchanges the two period loops. All of this is merely a game of Tannakian sudoku, and essentially boils down to the following consequence of Quillen's computation of K-groups of finite fields.

\begin{sRemark}
Let $\TM(Z) \subset \MTM(Z)$ be the full subcategory of semisimple objects. The composition
\[
\TM(Z) \subset \MTM(Z) \to \MTM(k)
\]
is an equivalence of categories, both canonically equivalent to the category of finite dimensional representations of $\Gm$ over $\QQ$. 
\end{sRemark}

We will use the following general Tannakian result, used also by Ancona and \Fratila.

\begin{sProposition}
\label{A10}
Let $F$ be a field, let $G$ be an affine group-scheme over $F$, and let $H < G$ be a closed subgroup. Let
\[
\om: \Rep_F G \to \Vect F
\]
be the canonical fiber functor and let
\[
\Inv_H : \Rep_F G \to \Vect F
\]
be the functor 
\[
V \mapsto V^H.
\]
Then the evident map
\[
G/H \to \Hom^\otimes(\Inv_H, \om) 
\] 
factors through an isomorphism
\[
\Spec \Oo(G)^H = \Hom^\otimes(\Inv_H, \om).
\]
\end{sProposition}

\begin{proof}
The proof of Proposition 2.3 of Haines \cite{hainesdrinfeld} applies \textit{mutatis mutandis} with `point' replaced by `$R$-point for $R$ an arbitrary $F$-algebra' throughout.
\end{proof}

\spar{A8}
There's an evident action the Tannakian Galois group $G_\dR(Z)$ (\ref{A3}*) on $H(Z)$. The Berthelot isomorphism (\ref{B2}**) defines a $\otimes$-compatible natural transformation
\[
\pP^\m{Ber}: e_{K_\pP/K_\pP^0} \circ \Cr^* \circ s_\pP^* \to e_{K_\pP/K} \circ \dR^*,
\]
the \emph{Berthelot path}. We note that the source of the Berthelot path is not exact, hence not a fiber functor. Denote the \emph{crystalline fundamental group} of $\MTM(Z)$ by
\[
G_{\Cr}(Z) = \Aut^\otimes (\Cr^* \circ s_\pP^*).
\]
The Berthelot path gives rise, on the one hand, to a map of affine $K_\pP$-groups
\[
\tag{**}
\si_\m{Cr}:G_{\Cr}(Z)_{K_\pP} \to G_\dR(Z)_{K_\pP},
\]
and on the other hand, to an orbit map 
\begin{align*}
o(\cber): G_\dR(Z)_{K_\pP} &\to H(Z)_{K_\pP}
\\
\ga &\mapsto \ga \cber
\end{align*}
(see (\ref{B2}) for the Berthelot point $\cber$). Evidently, $G_\m{Cr}(Z)_{K_\pP}$ is contained in the stabilizer, so we obtain a map
\[
\tag{*}
G_\dR(Z)_{K_\pP}/G_\m{Cr}(Z)_{K_\pP} \to H(Z)_{K_\pP}.
\]

\begin{sProposition}
\label{A9}
The map \ref{A8}(*) is an isomorphism. 
\end{sProposition}

\begin{proof}
Under the equivalence
\[
\MTM(Z)_{K_\pP} \simeq \Rep G_\dR(Z)_{K_\pP},
\]
the functor $\ze \circ s_\pP^*$ (\ref{A1}) corresponds to $\Inv_{G_\m{Cr}(Z)_{K_\pP}}$. So this follows from Proposition \ref{A10}.
\end{proof}

\begin{sProposition}
\label{A11}
The map $U_\dR(Z)_{K_\pP} \to H(Z)_{K_\pP}$ induced by the orbit map $o(\cber)$ (\ref{A8}) is an isomorphism. 
\end{sProposition}

\begin{proof}
We have $G_\m{Cr}(Z) = \GG_{m, K_\pP^0}$. Moreover, the map $\si_\m{Cr}$ (\ref{A8}**) splits the canonical surjection
\[
G_\dR(Z)_{K_\pP} \to \GG_{m, K_\pP}.
\]
So this follows from Proposition \ref{A9}.
\end{proof}

\begin{sProposition}
\label{A12}
The action of the Tannakian Galois group $G_\dR(Z)$ on the presheaf $H(Z)$ (defined over the number field $K$) makes $H(Z)$ into a trivial $U_\dR(Z)$-torsor. In particular, $H(Z)$ is representable by an affine $K$-scheme.  
\end{sProposition}

\begin{proof}
By Proposition \ref{A11}, the map (action, projection)
\[
U_\dR(Z) \times H(Z) \to H(Z) \times H(Z)
\]
is fpqc-locally an isomorphism, hence itself an isomorphism. This means that $H(Z)$ is an fpqc torsor under a prounipotent group over a field of characteristic zero, hence, a posteriori, a trivial torsor. 
\end{proof}

\spar{A13}
Recall from paragraph \ref{A3} that categories of mixed Tate motives admit canonical $\QQ$-rational fiber functors. We may undangerously regard the triangle
\[
\xymatrix{
\MTM(Z) \ar[r]^-{s_\pP^*} \ar[dr]_-{\om_\m{can}} 
& \MTM(k) \ar[d]^-{\om_\m{can}}
\\ & \Vect \QQ
}
\]
as commuting on the nose. There's an evident injective $\otimes$-compatible natural transformation 
\[
\tag{*}
\iota: \ze \subset \om_\m{can}
\]
(see \ref{A1} for $\ze$). 

\begin{Definition*} 
We define the \emph{de Rham point} $\cC_\dR: \Spec K \to  H(Z)$ by composing $\iota$ with the arithmetic Hodge path
\[
\pP^H: e_{K/\QQ} \circ \om_\m{can} = 
\dR^* \circ \gr^W \to \om_\dR
\]
of paragraph \ref{A6}. 
\end{Definition*}

\begin{proof}[Proof of Theorem \ref{A7}]
If instead we compose $\iota$ (\ref{A13}*) with the arithmetic crystalline path 
\[
s^\pP_*\pP^\m{Cr}_{K_\pP}: e_{K_\pP / \QQ} \circ \om_\m{can}
= e_{K_\pP/K } \circ \dR^* \circ \gr^W
\to 
e_{K_\pP/K } \circ \dR^*
\]
(transported to the de Rham fiber functor on $\MTM(Z)$ as in paragraph \ref{B3}), we find that
\[
s_*^\pP\pP^\m{Cr}_{K_\pP} \circ \iota = \cber
\]
(\ref{B2}). To see this, fix $M \in \MTM(Z)$, base change everything to $K_\pP$ and drop $K_\pP$ from the notation, and note that the diagram
\[
\xymatrix
{
& \om_\m{can}(M) 
\ar@{=}[r] \ar@{=}[d]
& \gr^W M_\dR 
\ar@{=}[d] \ar[r]^-{s^\pP_* \pP^\m{Cr}}
& M_\dR
\\
\Hom(\QQ(0), \gr^W_0 M_\pP)
\ar[r] \ar@{=}[d]
& \om_\can(M_\pP) \ar@{=}[r]
& \gr^W M_\m{Cr} \ar@{=}[d] \ar[r]^-{\pP^\m{Cr}}
& M_\m{Cr} \ar[u]_-{\pP^\m{Ber}}
\\
\Hom(\QQ(0), M_\pP) \ar[rr]
&& M_\m{Cr} \ar@{=}[ur]
}
\]
commutes. It follows that the classical $p$-adic period loop $\uU \in U_\dR(Z)(K_\pP)$ (\ref{B4}), the $K$-point $\cC_\dR \in H(Z)(K)$ (\ref{A13}), and the $K_\pP$-point $\cber \in H(Z)(K_\pP)$ (\ref{B2}) satisfy
\[
\uU \cC_\dR = \cber. 
\]
Consequently, the orbit map
\[
o(\cC_\dR): U_\dR(Z) \to H(Z)
\]
(defined over the number field $K$), fits in a commuting triangle
\[
\xymatrix
{
& H(Z)
\\
\Spec K_\pP \ar[ur]^-{\cber} \ar[dr]_-{\uU}
\\
& U_\dR(Z) \ar[uu]_-{o(\cC_\dR)}.
}
\]
Finally, by Proposition \ref{A12}, $o(\cC_\dR)$ is an isomorphism. 
\end{proof}

\section{Frobenius-fixed paths}

Our goal here is to explain the statement that the Frobenius-fixed paths of Besser and Vologodsky are motivic. This section presumes some familiarity with the formalism of ``algebraic geometry in a Tannakian category'' \cite{Deligne89} and the unipotent fundamental group \cite{DelGon}. 

\spar{A16}
Let $ \overline X \to Z$ be smooth and proper and let $X \subset \overline X$ be the complement of a snc divisor.  Consider points $x, y \in X(\hat \Oo_\pP)$ with values in the formal completion of the local ring at $\pP$ (\ref{Not}). Let $\opnm{UnVIC}(X_{K_\pP})$ be the $K_\pP$-Tannakian category of unipotent vector bundles with integrable connection, and let $\om_x$, $\om_y$ be the fiber functors
\[
\opnm{UnVIC}(X_{K_\pP}) \to \Vect K_\pP
\]
associated to $x,y$. There's a Frobenius action on $\Isom^\otimes (\om_x, \om_y)$; by Besser \cite{Besser} and Vologodsky \cite{Vologodsky}, $\Isom^\otimes (\om_x, \om_y)$ contains a unique Frobenius-fixed $K_\pP$-point $p^\m{Cr}$. 

\spar{A17}
When $X$ is of mixed Tate type, $x,y$ are $Z$-points of $X$, and 
\[
\om_x, \om_y: \opnm{UnVIC}(X_K) 
\rightrightarrows \Vect K
\]
are the associated fiber functors over $K$, the path torsor $\Isom^\otimes(\om_x, \om_y)$ is motivic. Thus, there's a prounipotent group object $\pi_1^\un(X,x)$ in $\MTM(Z)$, a torsor object $\pi_1^\un(X,x,y)$, and an isomorphism  
\[
\tag{$\dagger$}
\dR^*\pi_1^\un(X,x,y) = \Isom^\otimes(\om_x, \om_y).
\]
We may then consider the specialization $s^*_\pP\pi_1^\un(X,x,y)$, a torsor object in $\MTM(k)$, as well as the crystalline realization of the latter $\m{Cr}^* s_\pP^* \pi_1^\un(X,x,y)$, which carries the structure of a torsor under a prounipotent group in the category of $\phi$-modules over $K_\pP^0$. The Berthelot isomorphism
\[
\tag{*}
\dR^* \pi_1^\un(X,x,y)_{K_\pP} = \m{Cr}^* s_\pP^* \pi_1^\un(X,x,y)_{K_\pP}
\]
is compatible with the Frobenius actions. (Indeed, in the mixed Tate setting, this may be taken as a \textit{definition} of the Frobenius action on the left hand side; the arguments in Besser \cite{Besser} go through without the need to compare this definition with the one given there.)

\begin{sProposition}
\label{A18}
In the situation and the notation of paragraph \ref{A17}, there exists a unique morphism of affine scheme objects in $\MTM(k)$
\[
\tag{**}
\one \to s_\pP^* \pi_1^\un(X,x,y)
\]
whose crystalline realization maps via the Berthelot isomorphism \ref{A17}(*) to the Frobenius-fixed path $p^\m{Cr}$. 
\end{sProposition}

\begin{proof}
Any torsor object under a prounipotent group object in a semisimple Tannakian category is trivial, i.e. admits a point, as in (**). The crystalline realization is a map of $\phi$-modules, hence corresponds to a $\phi$-fixed point. Uniqueness then follows from the faithfulness of the realization functor.
\end{proof}

\spar{V1}
\textbf{Variant.}
When the theory of tangential base-points is available, the same holds with $Z$-integral tangential base-points in place of $x$ or $y$.

\spar{A19}
Beyond the mixed Tate setting, one expects Tannakian categories of mixed motives $\opnm{MM}(Z)$, and $\opnm{MM}(k)$, as well as a specialization functor
\[
s_\pP^*: \opnm{MM}(Z) \to \opnm{MM}(k)
\]
between them. Moreover, $\opnm{MM}(k)$ is expected to be semisimple. Standard expected properties would also lead to an associated theory of unipotent fundamental groups in close analogy with \ref{A17}.  

\begin{sProposition}
\label{conj}
In the situation and the notation of paragraph \ref{A16}, suppose $x,y$ are $Z$-points of $X$. Assume conjectures as indicated in paragraph \ref{A19}. Then there exists a unique morphism of affine scheme objects in $\opnm{MM}(k)$
\[
\one \to s_\pP^* \pi_1^\un(X,x,y)
\]
whose crystalline realization maps via the Berthelot isomorphism to the Frobenius-fixed de Rham path $p^\m{Cr}$. 
\end{sProposition} 

\begin{proof}
The proof of Proposition \ref{A18} applies with no change.
\end{proof}

\noindent
In fact, one may well expect Proposition \ref{conj} to extend to include the special de Rham paths identified by Vologodsky \cite{Vologodsky} in the semistable-reduction case (see also Betts-Litt \cite{betts2019semisimplicity}). 

\section{Example: $p$-Adic multiple polylogarithms}

Finally, in this section, we construct certain functions $I_b^c (\om_1 \cdots \om_m) \in \Oo H(Z)$ on the principal homogeneous space $H(Z)$ (\ref{extscal}**, \ref{A12}) which correspond under the equivalence (\ref{A7}) to Goncharov's motivic iterated integrals \cite{GonGal}; these are given by certain matrix entries of motivic path torsors. Moreover, we show that their Andr\'e periods are given by special values of Coleman functions. Thus, loosely speaking, \textit{Andr\'e periods of motivic path torsors are given by special values of Coleman functions}. 

\spar{A20}
Let $X$ be the complement of a family of disjoint sections $\infty, a_1, \dots, a_n$ of $\PP^1_Z \to Z$ and let
\[
\eta_i = \frac{dt}{t-a_i}.
\]
Fix a sequence $\om_1, \dots, \om_m$ of $1$-forms chosen from among the forms $\eta_i$, fix $Z$-points $b, c \in X(Z)$, and fix a prime $\pP \in Z$. The $\pP$-adic iterated integral
\[
\tag{*}
\int_b^c \om_1 \cdots \om_m \in K_\pP
\]
may be defined as follows. The de Rham path torsors are equipped with a \emph{nonabelian Hodge filtration}. A study of the structure of the Tannakian category of unipotent vector bundles with integrable connection on $X_{K}$ carried out by Deligne \cite{Deligne89} shows, on the one hand, that the de Rham path torsor \ref{A17}($\dagger$) is trivialized by the unique point $p^H$ in Hodge filtered degree $0$, and, on the other hand, that $\dR^*\pi_1^\un(X,x)$ is free prounipotent with a canonical set of free generators corresponding to monodromy about each of the punctures $a_1, \dots, a_n$. (See e.g. \cite{mtmue} for a quick summary of relevant facts concerning free prounipotent groups.) It follows that the word $\om = \om_1 \cdots \om_m$ gives rise to a function $f_\om \in \Oo\big(\dR^*\pi_1^\un(X,x,y)\big)$. By definition, 
\[
\int_b^c \om_1 \cdots \om_m = f_\om(p^\m{Cr}).
\]

\spar{A21}
\textbf{Variant.} We allow the upper and lower limits $b$ and $c$ to be $Z$-integral tangent vectors along the punctures $a_i$. See Deligne \cite{Deligne89} for tangential fiber functors, Deligne-Goncharov \cite{DelGon} for motivic path torsors between $Z$-integral tangent vectors, and e.g. Furusho \cite{FurushoPMZVII} for a discussion of Frobenius-fixed paths between tangential fiber functors. 

\spar{A24}
\textbf{Variant.} We allow the upper limit to vary over $X(\hat \Oo_\pP)$. The theory of Besser-Coleman functions \cite{Besser} applies to define a locally analytic function
\[
\int_b^x \om_1 \cdots \om_m.
\]

\spar{A23}
Special cases of the above go by special names. When $X = \thrpl$ and $b = \vec{1}_0$ is the tangent vector ``$1$ at $0$'', then
\[
\int_{\vec{1}_0}^x \om_1 \cdots \om_m
\]
is known as a \emph{$p$-adic multiple polylogarithm}. When, moreover, the upper limit is taken to be $c = -\vec{1}_1$, the integral
\[
\int_{\vec{1}_0}^{-\vec{1}_1} \om_1 \cdots \om_m
\]
is known as a \emph{$p$-adic multiple zeta value}.

\medskip
The $\pP$-adic iterated integral \ref{A20}(*) is known to be in the image of the classical $\pP$-adic period map \ref{B4}(*), so by Theorem \ref{A7}, it's also in the image of the Ancona-\Fratila period map \ref{B2}($\dagger$). Our goal here is just to spell this out explicitly, and to point out the connection with our observation in Proposition \ref{A18} that the Frobenius-fixed de Rham paths are motivic. 

\spar{A25}
Let $\Uu(X,y)$ be the completed universal enveloping algebra of $\pi_1^\un(X,y)$. By pushing out along the inclusion
\[
\pi_1^\un(X,y) \to \Uu(X,y),
\]
we may turn the path torsor $\pi_1^\un(X,x,y)$ into a left $\Uu(X,y)$-module, which we denote by $\Uu(X,x,y)$. 
The motivic path of Proposition \ref{A18} gives rise to morphism
\[
p^\m{mot}: \one \to s_\pP^*\Uu(X,x,y)
\]
in the pro-category $\opnm{Pro} \MTM (k)$. The data
\[
I_b^c (\om_1 \cdots \om_m) :=
[\Uu(X,x,y), p^\m{mot},  f_\om]
\]
induces a function 
\begin{align*}
H(Z) = \Hom^\otimes (\ze \circ s_\pP^*, \dR^*)
&\to \AA^1_K
\\
\cC 
&\mapsto I_b^c (\om_1 \cdots \om_m)(\cC)
\end{align*}
defined on points (with values in an arbitrary $K$-algebra which we take to be $K$ only in order to lighten notation) in analogy with Tannakian matrix entries: $p^\m{mot}$ is an element of $\ze\Big(s_\pP^*\big(\Uu(X,x,y)\big)\Big)$ (\ref{A1}**); applying $\cC$, we get a point $\cC p^\m{mot}$ of $\dR^*\Uu(X,x,y)$; the function $f_\om$ extends to a linear functional on $\dR^*\Uu(X,x,y)$; we define
\[
I_b^c (\om_1 \cdots \om_m)(\cC)
:=
f_\om \big( \cC p^\m{mot} \big).
\]

\begin{sProposition}
In the situation and the notation of paragraph \ref{A25}, 
\[
\m{per}^\m{AF}\big( I_b^c(\om_1 \cdots \om_m) \big) 
=
\int_b^c \om_1 \cdots \om_m.
\]
\end{sProposition}

\begin{proof}
We have
\begin{align*}
\m{per}^\m{AF}\big( I_b^c(\om_1 \cdots \om_m) \big)
&=
I_b^c(\om_1 \cdots \om_m)(\cber)
\\
&=
f_\om(\cber p^\m{mot})
\\
&=
f_\om(p^\m{Cr})
\\
&= 
\int_b^c \om_1 \cdots \om_m. \qedhere
\end{align*}
\end{proof}

\bibliography{AF_refs}
\bibliographystyle{alpha}

\vfill

\Small

\noindent
\textsc{Ishai Dan-Cohen} 

 \noindent
\textsc{Email address:} \url{ishaida@bgu.ac.il}

\noindent
\url{https://ishaidcgit.github.io/}

\end{document}